\documentclass[12pt]{amsart}

\usepackage{amsmath}
\usepackage{amssymb}
\usepackage{amsthm}
\usepackage{amsfonts}
\usepackage{graphicx}
\usepackage{color}
\usepackage{mathrsfs}

\newtheorem{thm}{Theorem}[section]

\providecommand{\D}{\Delta}

\providecommand{\e}{\varepsilon}

\providecommand{\s}{\sigma}

\newcommand{\sG}{\mathscr{G}}

\begin{document}

\author{Thomas M. Lewis}
\address{Furman University,
Greenville, SC, USA}
\email{tom.lewis@furman.edu}


\thanks{I would like to dedicate this article to  Stephen  Hedetniemi,  
not  only for introducing  me to the  Hamiltonian  number  problem,  
but for his unfailing encouragement along the way.
}

\title{On the Hamiltonian Number of a Planar Graph}


\maketitle

\begin{abstract}
The Hamiltonian number  of a connected  graph is the minimum  of the 
lengths of the closed, spanning  walks in the graph.  In 1968, 
Grinberg published a necessary condition for the existence of 
a Hamiltonian cycle in a planar  graph,  formulated in terms  
of the lengths  of its face cycles.  We show how Grinberg's 
theorem  can be adapted to provide a lower bound  on the
Hamiltonian number  of a planar graph.
\end{abstract}

\section{Introduction}
A \emph{Hamiltonian  cycle} in a graph  is a closed, spanning  walk that  
visits each vertex exactly once; a graph is called \emph{Hamiltonian} 
provided that  it contains a Hamiltonian  cycle. 
While not every graph is Hamiltonian,  
every connected graph contains a closed, 
spanning walk. A closed, spanning walk of minimum 
length is called a \emph{Hamiltonian  walk}, and 
the \emph{Hamiltonian  number}  of a connected
graph $G$, denoted  by $h(G)$,  
is the length of a Hamiltonian  walk in $G$.  
Thus the  Hamiltonian  number  of a graph  can be understood  
as a measure  of how far the graph deviates from being Hamiltonian.

In 1968, Grinberg  \cite{Grinberg68} published  a necessary condition 
for the existence of a Hamiltonian  cycle in a planar  graph, formulated 
in terms of the lengths of its face cycles. The sole purpose of this paper 
is to show how Grinberg's theorem can be adapted  to provide a lower 
bound on the Hamiltonian number of a planar  graph. 
Before we state  this  theorem,  it is helpful to place our work in context.

In general, determining  the  Hamiltonian  number  of a graph  is difficult, 
but  for a connected graph $G$ of order $n$, the elementary  bounds
\[
n \le h(G) \le 2(n - 1)
\]
are easily obtained.  A Hamiltonian  walk on $G$ must visit 
each vertex, which gives the lower bound.  
On the other hand, a pre-order, closed, spanning walk 
on a spanning tree of $G$ has length $2(n - 1)$, 
yielding the upper bound.  Over the  years,  
much  of the  research  on the  Hamiltonian  number  
has advanced along two fronts:  developing tighter  
bounds for the Hamiltonian  number  in terms  of natural 
graph parameters, or evaluating  the Hamiltonian  numbers 
of some special graphs or families of graphs.

Goodman and Hedetniemi \cite{GoodmanEtAl74}
initiated  the study of the Hamiltonian number 
of a graph.  They proved, among other things, properties  of 
Hamiltonian walks, upper and lower-bounds for the Hamiltonian 
number of a graph,  and a formula for the Hamiltonian  number 
of a complete $n$-partite graph.  Their most  accessible result  
is this:  let  $G$ be a $k$-connected  graph  on $n$  vertices with 
diameter $d$; then
\[
h(G) \le 2(n - 1) - \lfloor k/2 \rfloor (2d - 2),
\]
which improves the elementary  upper bound.  
In another  result, they related the Hamiltonian numbers 
of $G$ and $G-U$ , the graph  obtained  by 
deleting the unicliqual points of G.

Soon after the publication  of the seminal paper of Goodman and 
Hedetniemi, Bermond \cite{Bermond76} published a theorem on the Hamiltonian  
number problem inspired  by Ore's theorem.   Ore's theorem  gives 
a sufficient condition for a graph to be Hamiltonian  in terms 
of the sums of the degrees of non-adjacent vertices;  
see, for example, Theorem 6.6 of 
\cite{Chartrand12}. Bermond showed the following: let $G$ be 
a graph of order
$n$ and let $c \le n$; if $\deg(v) + \deg(w) \ge c$ 
for every pair of non-adjacent vertices
$v$ and $w$ in $V (G)$, then  $h(G) \le 2n - c.$

Chartrand, Thomas,  Zhang,  and  Saenpholphat 
\cite{ChartrandEtAl04} introduced  
an alternative  approach  to the  Hamiltonian  number.   
Let $G$ be a connected  graph of order $n$.  
Given vertices $u$ and $v$,  let $d(u, v)$ denote the length 
of a shortest path  from $u$ to $v$.  
A cyclic ordering  of the  vertices  of $G$ is a permutation 
$s : v_1 , v_2, \ldots , v_n , v_{n+1}$  of $V (G)$,  where $v_{n+1}=v_1$.  
Given a cyclic ordering $s$, let $d(s) = \sum_{i=1}^n d(v_i, v_{i+1})$.
The set
\[
H (G) = \{ d(s) : \text{$s$ is a cyclic ordering of $V (G)$} \}
\]
is called  the  \emph{Hamiltonian spectrum} of $G$.  
Chartrand and  his  colleagues showed that   
$h(G)  = \min H (G).$    
This  paper  contains  two  other  notable results:  
first, that  a connected  graph $G$ of order $n$ satisfies 
$h(G) \le 2(n - 1)$ with equality  if and only if $G$ is a tree; 
second, that  for each integer $n \ge 3$, every integer in the 
interval  $[n, 2(n - 1)]$ is the Hamiltonian  number of some 
graph of order $n$. Kr{\'a}l, Tong, and Zhu \cite{KralEtAl06} 
and Liu \cite{Liu11} conducted additional research on the 
Hamiltonian spectra of graphs.

Various authors  have studied  the Hamiltonian  number  of 
special graphs and families of graphs.  Punnim  and Thaithae 
\cite{ThaithaeEtAl08, PunnimEtAl10}
studied  the Hamiltonian numbers of cubic graphs.  
A graph of order $n$ with Hamiltonian  number $n+1$ is called 
\emph{almost Hamiltonian}.  
Punnim,  Saenpholphat, and Thaithae \cite{PunnimEtAl07}
characterized  the  almost  Hamiltonian  cubic graphs  and  
the almost  Hamiltonian  generalized Petersen  graphs.  
Asano, Nishizeki, and Watanabe \cite{AsanoEtAl80, AsanoEtAl83}
established  a simple upper bound for the Hamiltonian  number of 
a maximal planar graph  of order $n  \ge 3$ and  created  an  
algorithm  for finding closed, spanning  walks in  a  graph  
with  length  close to  its  Hamiltonian  number. 
G.~Chang,  T.~Chang,  and  Tong \cite{ChangEtAl12}
studied  the  Hamiltonian  numbers  of M{\"o}bius double loop networks.
The Hamiltonian  number  problem has a variety  of cognates:  
Vacek \cite{Vacek91, Vacek92}
analyzed open Hamiltonian  walks; 
Araya and Wiener \cite{WienerEtAl11a, WienerEtAl11b}
investigated hypohamiltonian graphs;  
Goodman,  Hedetniemi,  and Slater 
\cite{GoodmanEtAl75} studied  the the Hamiltonian  
completion problem; T. Chang and Tong \cite{ChangEtAl13}
considered the Hamiltonian 
numbers of strongly connected digraphs; and, Okamoto,  Zhang, and 
Saenpholphat \cite{OkamotoEtAl06, OkamotoEtAl08}
studied  the upper traceable  numbers of graphs.

\section{The Grinberg number of a planar graph}

Let $G$ be a planar  graph  and  let the  faces 
(including  the  unbounded  component) 
be labeled $F_1 , F_2, \ldots , F_N$. 
Given a face $F$, let $|F|$ denote its length, 
that  is, the number of its edges (or vertices).  
Let $\sG (G)$ be the set of all sums of the form
\[
\left| \sum_{i=1}^N \e_i (|F_i|-2) \right|
\]
where $(\e_1 , \e_2, \ldots , \e_N ) \in  \{1, +1\}^N$, 
excluding the two cases where all of the
components  have the same sign.  
Let $g(G) = \min \sG (G).$  will call $\sG (G)$
the Grinberg set of $G$ and $g(G)$ the Grinberg number  
of $G$. To develop a feel for this, consider the graph 
presented in Figure \ref{fig:honeycomb}. 
This graph has five interior faces, 
each a hexagon, and the exterior face has 18 edges. 
The Grinberg set for this graph is $\{4, 12, 20, 28\}$ 
and the Grinberg number is 4.

\begin{figure}[ht.]
\def\svgwidth{.5\textwidth}
\centering
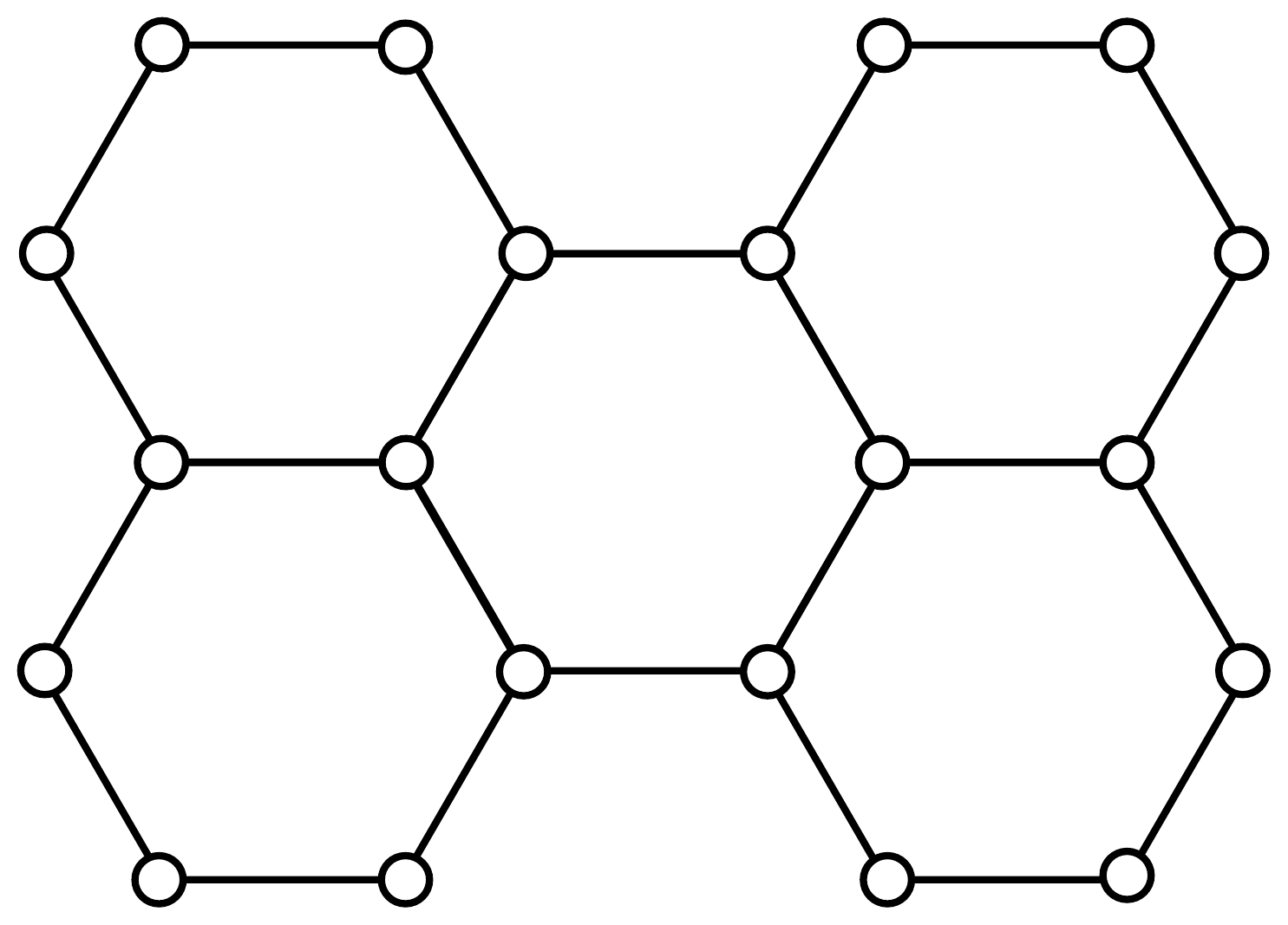
\caption{A graph with Grinberg set $\{4, 12, 20, 28\}$
an Grinberg number 4.}
\label{fig:honeycomb}
\end{figure}

Grinberg's theorem can be stated as follows:
if a planar graph $G$ is Hamiltonian,  
then $g(G) = 0$; see \cite{Grinberg68}. 
Our  main result  can be seen as a natural extension  
of Grinberg's theorem.

\begin{thm}\label{thm:Grinberg}
Let G be a planar  graph. 
A closed, spanning walk on G must contain 
$g(G)/2$ repeats of vertices.
\end{thm}

\begin{proof} Our proof is an adaptation of the customary proof
of Grinberg's theorem; see, for example, Theorem 18.2 of \cite{Bondy08}.
Let the vertices of the planar  graph $G$ be labeled 
$\{v_1 , v_2, . . . , v_n\}$ and let $\s$ be a closed, 
spanning walk on $G$, given as a list of these vertices.  
First we will produce a new planar  multigraph  $G'$, 
called the \emph{reduction} of $G$, relative to $\s$. 
The reduced graph is created  through  applications  
of the following procedures:
\begin{description}
\item[Edge removal] Remove any edge from $G$ that  was not 
traversed  by $\s$.
\item[Edge duplication] If an edge of $G$ is traversed  more than 
once by $\s$, then create an additional  edge in $G'$ for each additional  traversal of this edge by this walk.
\end{description}
Thus, while the reduced graph $G'$ is still planar,  
it may no longer be simple. The faces of $G'$  
are labeled with a $+$ or a $-$ sign as follows: 
the unbounded component is marked  $+$;  
thereafter, if a face of $G'$  is adjacent to 
(shares an edge with)  a $+$ region, then  it is marked $-$, 
and if a face of $G'$ is adjacent to a $-$ region, then  it is 
marked $+$. An example of a planar  graph  and its reduction 
relative to a closed, spanning walk is presented 
in Figure \ref{fig:GReduce}.

\begin{figure}
\def\svgwidth{.9\textwidth}
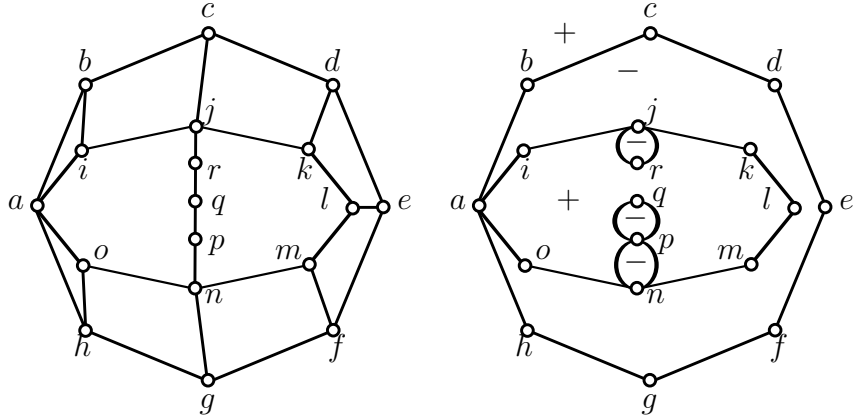
\caption{A  planar  graph (left) and  its  reduction (right) based on the 
closed, spanning  walk $a, b, c, d, e, f, g, h, a, i, j, r, j, k, l, m, 
n, p, q, p, n, o, a$. Notice that the reduced graph has 4 repeated vertices and 6 faces.}
\label{fig:GReduce}
\end{figure}

There is a simple relationship between the number of faces of $G'$
and the number  of repeats  of vertices  in the  closed, spanning  walk $\s$. 
For  each $i$, $1 \le i \le n$, let $m_i$ count the number of times vertex 
$v_i$  is repeated  in the walk $\s$.
To be clear about this, if the vertex $v_i$ 
appears only once in the walk $\s$, then $m_i = 0$.  
Let $\Phi$ count the number  of faces of the reduced graph $G'$.  
By way of the Euler characteristic formula, we will show that
\begin{equation}\label{eq:FaceRepeats}
\Phi = 2 + \sum_{i=1}^n m_i.
\end{equation}
The degree of the vertex $v_i$ is $2m_i + 2$ and therefore the number 
of edges of $G'$ is half of the sum of the degrees taken over 
all vertices, that  is $\sum_{i=1}^n (m_i+1).$
Thus
\[
n - \sum_{i=1}^n (m_i+1) + \Phi = 2,
\]
from which formula (\ref{eq:FaceRepeats}) follows.

Our argument now moves into a second phase,  
culminating  in a simple formula relating  the 
lengths  of the positively and negatively  signed faces of $G'$.  
Let the faces of $G'$  be divided according to their signs and let 
$\{A^+  : 1 \le n_+ \}$
be the labels for the positively signed faces and let
$\{A^- : 1 \le n_- \}$
be the labels for the negatively signed faces. Since each edge of $G'$
is adjacent to a positively and a negatively signed face, it follows that
\[
\sum_{i=1}^{n_+} |A^+| = \sum_{i=1}^{n_-} |A^-|.
\]
Let $\D = n_- - n_+$, the difference between the number of negatively
and positively signed faces of $G'$. Then 
\begin{equation}\label{eq:SignedFaceDiff}
\left| \sum_{i=1}^{n_+} (|A^+|-2) - \sum_{i=1}^{n_-} (|A^-|-2)
\right| = 2 | \Delta|.
\end{equation}

We will modify this formula to incorporate  the faces of $G$. 
Let the faces of $G$ be 
labeled $\{F_i  : 1 \le i \le N\}$.  
Each face of $G$ is contained  by a unique face of $G'$.  
For each $i$, $1 \le i \le N$, let $\e_i$  be sign of the face of $G'$
that  contains  $F_i$. We will show that
\begin{equation}\label{eq:SignedFaceDiff2}
\left| \sum_{i=1}^N \e_i ( |F_i| - 2) \right| = 2 | \D |.
\end{equation}
To verify this  claim, we will follow Grinberg's strategy:  
we will add to $G'$ , one at a time, those edges of G that  
were not traversed  by $\s$. Such an edge must  split  a face of 
$G'$  into  two sub-faces, each with  the  same sign as the parent face.  
For the sake of argument,  let us say that  a face labeled $A_i^+$  is
divided by an edge of $G$ into two sub-faces, labeled $A_{i_1}^+$
and $A_{i_2}^+$.  Since the
two sub-faces share exactly one edge, we have
\[
|A_i^+| - 2 = (|A_{i_1}^+|-2) + (|A_{i_2}^+| - 2).    
\]                
Hence we can substitute $(|A_{i_1}^+|-2) + (|A_{i_2}^+| - 2)$
for $|A_i^+| - 2$ in equation (\ref{eq:SignedFaceDiff}) and retain equality.   
We continue  this  process until all of these edges have been
added. We have almost  arrived  at  equation 
(\ref{eq:SignedFaceDiff2}). The only difference corresponds  to those 
faces of $G'$ that  were created  because an edge was traversed  
more than  once by $\s$, the closed, spanning  walk.  
Such a face has only two edges and thus  contributes 0 in the sum.  
In this way, we have transformed  equation (\ref{eq:SignedFaceDiff}) 
into equation (\ref{eq:SignedFaceDiff2}). 

Our proof is nearly complete.  
By the definition of the Grinberg number
of a graph and equation (\ref{eq:SignedFaceDiff2}),
\[
g(G) \le \left| \sum_{i=1}^N \e_i ( |F_i| - 2) \right| = 2 | \D |
\]
Recall that $n_-$ and $n_+$ count the number of negatively and positively 
signed faces in $G'$, respectively, and each 
of these must be at least 1. Since they sum to 
$\Phi$, it must be that  $|\Delta| \le \Phi- 2$. 
Bearing in mind equation (\ref{eq:FaceRepeats}), we obtain
\[
\frac{1}{2} g(G) \le \sum_{i=1}^n m_i,
\]
as was to be shown.
\end{proof}

\section{Some examples}

In this section we present several examples of Theorem \ref{thm:Grinberg}.
\begin{enumerate}
\item The graph in Figure \ref{fig:honeycomb} 
has Grinberg number 4; thus, any closed, spanning walk 
on this graph must have two repeats of vertices.  
It is easy to find a closed, spanning walk with two repeats 
and thus this is optimal.
\item Consider the simple grid graph presented in Figure 3. 
The exterior face has 8 edges; each of the four interior  
faces has 4 edges. The Grinberg number of this  graph  
is 2, and any  closed, spanning  walk on this graph must contain  
at least one repeated vertex. It is easy to 
find an example of a closed, spanning  walk with one repeated vertex.
\begin{figure}[htb]
\centering
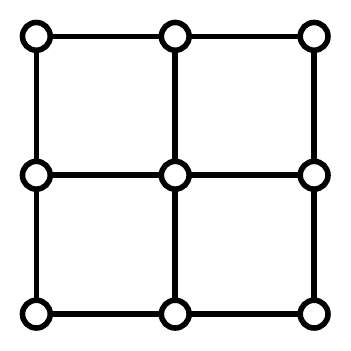
\caption{A simple grid graph with Grinberg number 2.}
\label{fig:GridGraph}
\end{figure}
\item The Grinberg  number  of a tree can be computed  by 
duplicating  each edge and computing the Grinberg number 
of the resulting planar graph. The altered  version of the 
tree in Figure \ref{fig:tree} has an outside face with 20 edges and 10 
interior faces with 2 edges each. The Grinberg number of this 
tree is 18 and any closed, spanning walk on this tree 
must contain at least 9 repeats  of vertices.
\begin{figure}
\centering
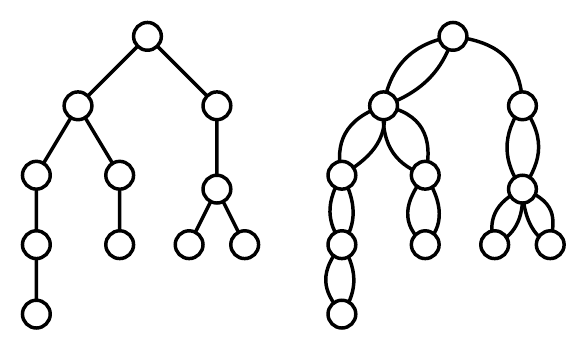
\caption{A tree and its alteration.}
\label{fig:tree}
\end{figure}

\item The graph presented  in Figure \ref{fig:14Simple} has an 
outside face with 26 edges and three  interior  faces, 
each with 14 edges.  The  Grinberg  number  is 12; thus,  
any  closed, spanning  walk on this  graph  must  
have  at  least  6 repeats.  
It is easy to find a closed, 
spanning  walk with 6 repeats.  
For example, begin at $a$ and walk around  
the outside of the graph,  ending at $a$.   
Now augment  
the  walk by adding  $a, b, c, d, c, b, a, e, f, g, f, e, a$. 
This walk has 6 repeats  of vertices:  $a$ (twice), $b$, $c$, $e$, and $f$.
\begin{figure}[htb]
\def\svgwidth{.5\textwidth}
\centering
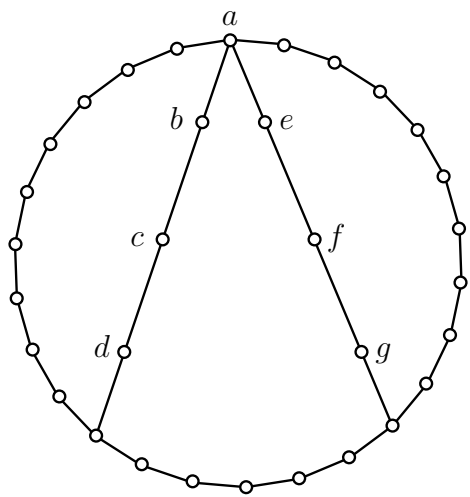
\caption{A graph with Grinberg number 12. A closed, spanning walk on this
graph must have at least 6 repeats of vertices.}
\label{fig:14Simple}
\end{figure}

\item Figure \ref{fig:58Simple} exhibits a graph with 8 
interior faces, each an octagon, and an exterior  
face with 20 edges.  The Grinberg  
number  of this  graph  is 6. Any closed, spanning  
walk on this graph must  have at least 3 repeats of vertices.  
A closed, spanning  walk with exactly 3 repeats  is pictured in the figure.
\begin{figure}[htb]
\def\svgwidth{.7\textwidth}
\centering
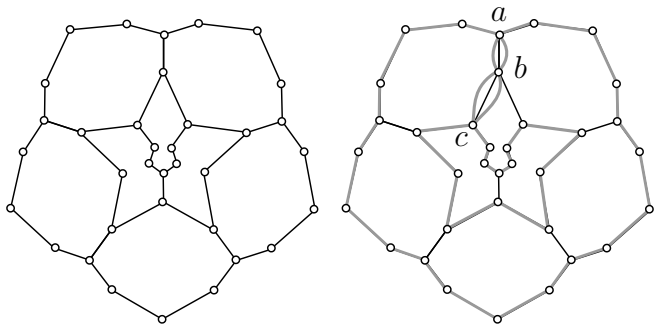
\caption{The graph on the left has Grinberg number 6.
A Hamiltonian walk on this graph (with 3 repeats) is shown on the right.}
\label{fig:58Simple}
\end{figure}

\end{enumerate}

\section{Additional observations}

Given a closed, spanning walk on the planar graph $G$, 
let $\rho$ count the number of repeats of vertices,
let $\nu = n^--1$, and let $\pi=n_+-1$.
In other words, $\nu$ is one less than the number of negatively signed
regions in $G'$ and $\pi$ is one less than the number of its 
positively signed regions in $G'$.

Since the reduced graph must  have at least one of each type 
of each signed region, $\nu$ and $\pi$ are nonnegative.  
By equations (\ref{eq:FaceRepeats}) and (\ref{eq:SignedFaceDiff2}), 
there exists an element $f \in \sG (G)$ for which
\begin{align*}
f 				&= \nu + \pi	\\
\frac{1}{2} f 	&= \max \{\nu,\pi\} - \min \{\nu, \pi \}
\end{align*}
In particular, this shows that
\[
\rho = \frac{1}{2} f + 2 \min\{ \nu, \pi \}.
\]
This can be a helpful observation  when searching for Hamiltonian walks.  
For example, consider a planar graph $G$ with 8 interior faces, 
each an octagon, and  an exterior  face with  20 edges.  
The  Grinberg  set for this graph is $\{6, 18, 30, 42, 54\}$;
hence, the number of repeats of vertices in a 
Hamiltonian  walk on G must be at least 3 and odd.  
The graph pictured in Figure \ref{fig:58Simple} has a Hamiltonian
walk with the minimal number 
of repeats, while the graph pictured in Figure \ref{fig:Octo}
has a Hamiltonian walk that requires 5 repeats.
\begin{figure}[htb]
\centering
\includegraphics[scale=.4]{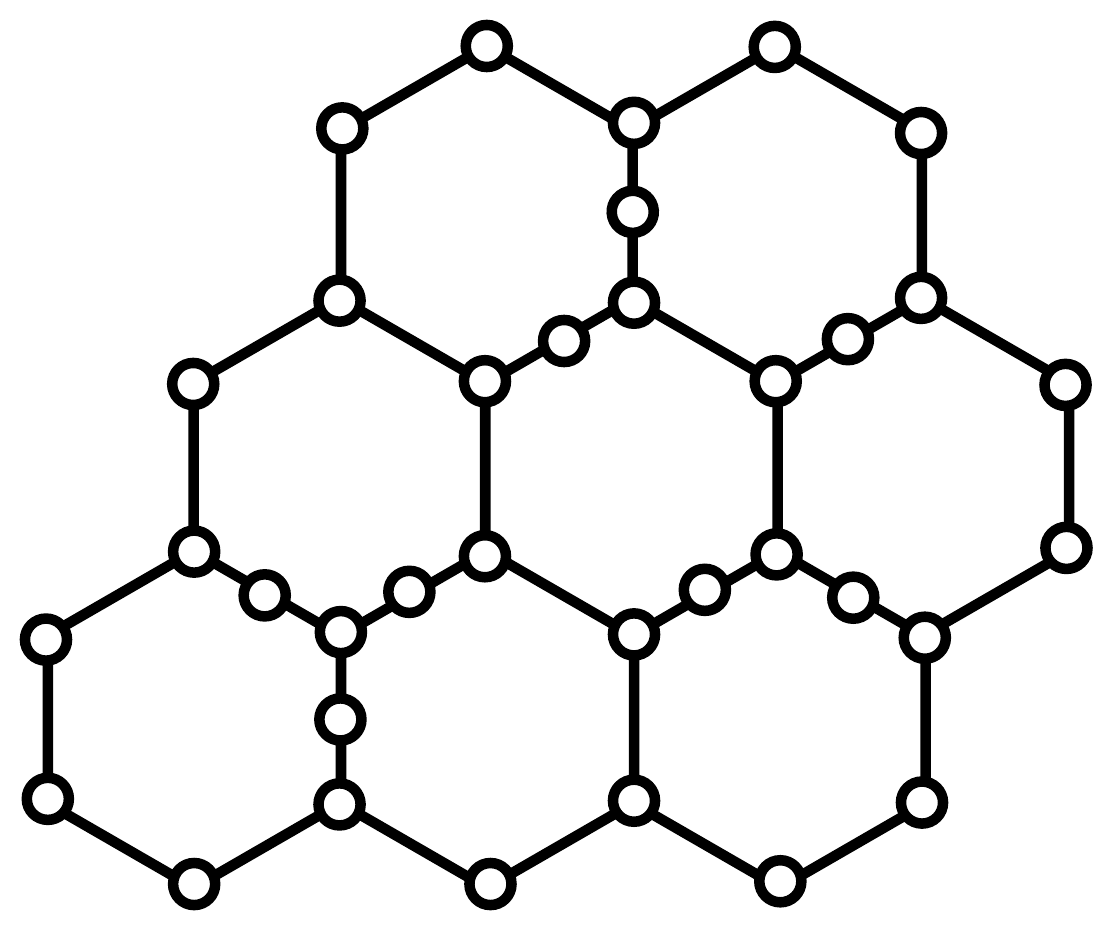}
\caption{A graph, composed of octagons, that  has a Hamiltonian  walk with
5 repeats.}
\label{fig:Octo}
\end{figure}

\bibliography{HamNum}{}
\bibliographystyle{plain}
\end{document}